\documentclass[12pt]{amsart}
\pdfoutput=1
\usepackage[paperheight=11in,paperwidth=8.5in,left=1in,top=1.25in,right=1in,bottom=1.25in]{geometry}
\usepackage[pdfstartview={FitH},pdftitle={},pdfauthor={},pdfsubject={},pdfcreator={},pdfproducer={},pdfkeywords={},pagebackref=false,colorlinks=true,linkcolor=black,linktoc=all,citecolor=black,filecolor=black,urlcolor=black]{hyperref}

\usepackage{amsmath,amssymb}
\usepackage[capitalize,nameinlink,noabbrev,nosort]{cleveref}
\usepackage[justification=centering]{caption}
\usepackage{enumerate}
\usepackage{enumitem}
\usepackage{tikz}\pdfpageattr{/Group <</S /Transparency /I true /CS /DeviceRGB>>}
\usepackage{todonotes}
\allowdisplaybreaks

\newtheorem{theorem}{Theorem}[section]
\newtheorem{lemma}[theorem]{Lemma}
\newtheorem{proposition}[theorem]{Proposition}
\newtheorem{conjecture}[theorem]{Conjecture}
\newtheorem{corollary}[theorem]{Corollary}

\newtheorem*{theorem*}{Theorem}

\theoremstyle{definition}
\newtheorem{definition}[theorem]{Definition}
\newtheorem{example}[theorem]{Example}

\newtheorem{remark}[theorem]{Remark}

\numberwithin{equation}{section}

\newcommand{\R}{\mathbb{R}}

\newcommand{\bool}[1]{B_{#1}}
\newcommand{\eulerian}{\genfrac{\langle}{\rangle}{0pt}{}}
\DeclareMathOperator{\Gr}{Gr}
\newcommand{\rankfun}[1]{\textnormal{rank}(#1)}
\newcommand{\sym}[1]{\mathfrak{S}_{#1}}

\begin{document}

\title{Shelling the $m=1$ amplituhedron}

\author{Steven N. Karp}
\address{Department of Mathematics, University of Notre Dame}
\email{\href{mailto:skarp2@nd.edu}{skarp2@nd.edu}}
\author{John Machacek}
\address{Department of Mathematics, University of Oregon}
\email{\href{mailto:johnmach@uoregon.edu}{johnmach@uoregon.edu}}

\thanks{S.N.K.\ was partially supported by an NSERC postdoctoral fellowship.}

\begin{abstract}
\hspace*{4pt} The amplituhedron $\mathcal{A}_{n,k,m}$ was introduced by Arkani-Hamed and Trnka \cite{arkani-hamed_trnka14} in order to give a geometric basis for calculating scattering amplitudes in planar $\mathcal{N}=4$ supersymmetric Yang--Mills theory. It is a projection inside the Grassmannian $\Gr_{k,k+m}$ of the totally nonnegative part of $\Gr_{k,n}$. Karp and Williams \cite{karp_williams19} studied the $m=1$ amplituhedron $\mathcal{A}_{n,k,1}$, giving a regular CW decomposition of it. Its face poset $R_{n,l}$ (with $l := n-k-1$) consists of all projective sign vectors of length $n$ with exactly $l$ sign changes. We show that $R_{n,l}$ is EL-shellable, resolving a problem posed in \cite{karp_williams19}. This gives a new proof that $\mathcal{A}_{n,k,1}$ is homeomorphic to a closed ball, which was originally proved in \cite{karp_williams19}. We also give explicit formulas for the $f$-vector and $h$-vector of $R_{n,l}$, and show that it is rank-log-concave and strongly Sperner. Finally, we consider a related poset $P_{n,l}$ introduced by Machacek \cite{machacek}, consisting of all projective sign vectors of length $n$ with at most $l$ sign changes. We show that it is rank-log-concave, and conjecture that it is Sperner.
\end{abstract}

\maketitle

\section{Introduction}\label{sec:intro}

\noindent Let $\Gr_{k,n}^{\ge 0}$ denote the {\itshape totally nonnegative Grassmannian} \cite{postnikov,lusztig94}, comprised of all $k$-dimensional subspaces of $\mathbb{R}^n$ whose Pl\"{u}cker coordinates are nonnegative. Motivated by the physics of scattering amplitudes, Arkani-Hamed and Trnka \cite{arkani-hamed_trnka14} introduced a generalization of $\Gr_{k,n}^{\ge 0}$, called the {\itshape (tree) amplituhedron} and denoted $\mathcal{A}_{n,k,m}(Z)$. It is defined as the image of $\Gr_{k,n}^{\ge 0}$ under (the map induced by) a linear surjection $Z : \mathbb{R}^n \to \mathbb{R}^{k+m}$ whose $(k+m)\times (k+m)$ minors are all positive. While the definition of $\mathcal{A}_{n,k,m}(Z)$ depends on the choice of $Z$, it is expected that its geometric and combinatorial properties only depend on $n$, $k$, and $m$. The amplituhedron may be regarded as a generalization of both a cyclic polytope (which we obtain when $k=1$) and the totally nonnegative Grassmannian $\Gr_{k,n}^{\ge 0}$ (which we obtain when $k+m = n$).

When $m=4$, the amplituhedron $\mathcal{A}_{n,k,4}(Z)$ gives a geometric basis for computing tree-level scattering amplitudes in planar $\mathcal{N} = 4$ supersymmetric Yang--Mills theory, but it is an interesting mathematical object for any $m$. In \cite{karp_williams19}, Karp and Williams carried out a detailed study of the $m=1$ amplituhedron $\mathcal{A}_{n,k,1}(Z)$. They gave a regular CW decomposition of $\mathcal{A}_{n,k,1}(Z)$, whose face poset, which we denote by $R_{n,l}$ (with $l := n-k-1$), can be described as follows.\footnote{For simplicity, we use an equivalent but slightly different labeling of the face poset than in \cite{karp_williams19}. Namely, in \cite{karp_williams19}, the face poset is denoted $\overline{\mathbb{P}\!\operatorname{Sign}_{n,k,1}}$, and is obtained from $R_{n,l}$ by applying the involution $(v_1, \dots, v_n) \mapsto (v_1, -v_2, v_3, -v_4, \dots, (-1)^{n-1}v_n)$.} The elements of $R_{n,l}$ are {\itshape projective sign vectors} of length $n$ (i.e.\ elements of $\{0,+,-\}^n\setminus\{(0,\dots, 0)\}$ modulo the relation $(v_1, \dots, v_n) \sim (-v_1, \dots, -v_n)$) with exactly $l$ sign changes. The order relation in $R_{n,l}$ is such that
\begin{align}\label{sign_vector_definition}
(v_1, \dots, v_n) \le (w_1, \dots, w_n) \quad \iff \quad v_i\in\{0,w_i\} \text{ for } 1 \le i \le n.
\end{align}
For example, $R_{3,1}$ is depicted in \cref{fig:R31}.

Karp and Williams posed the problem \cite[Problem 6.19]{karp_williams19} of showing that the poset $R_{n,l}$ is shellable. We resolve this problem:
\begin{theorem}\label{thm:shelling_intro}
The poset $R_{n,l}$ with a minimum and a maximum adjoined is EL-shellable.
\end{theorem}

The motivation behind \cite[Problem 6.19]{karp_williams19} was the following. Karp and Williams showed that the $m=1$ amplituhedron $\mathcal{A}_{n,k,1}(Z)$ is a regular CW complex which can be identified with the bounded complex of a certain generic arrangement of $n$ hyperplanes in $\mathbb{R}^k$ (namely, a {\itshape cyclic arrangement}). It then follows from a general result of Dong \cite{dong08} that $\mathcal{A}_{n,k,1}(Z)$ is homeomorphic to a $k$-dimensional closed ball. Karp and Williams observed that rather than appealing to \cite{dong08}, one could reach the same conclusion by showing that the face poset $R_{n,l}$ is shellable, using a result of Bj\"{o}rner \cite[Proposition 4.3(c)]{bjorner84}. (This relies on the regular CW decomposition, along with the fact that every cell of codimension one is contained in the closure of at most two maximal cells.) Therefore, as a consequence of \cref{thm:shelling_intro}, we obtain a new proof that $\mathcal{A}_{n,k,1}(Z)$ is homeomorphic to a closed ball:
\begin{corollary}[{\cite[Corollary 6.18]{karp_williams19}}]
The $m=1$ amplituhedron $\mathcal{A}_{n,k,1}(Z)$ is homeomorphic to a $k$-dimensional closed ball.
\end{corollary}

We expect that for any $m\ge 1$, the amplituhedron $\mathcal{A}_{n,k,m}(Z)$ has a shellable regular CW decomposition and is homeomorphic to a closed ball, thereby generalizing the situation which holds when $n = k+m$. Indeed, in this case $\mathcal{A}_{n,k,n-k}(Z)$ is the totally nonnegative Grassmannian $\Gr_{k,n}^{\ge 0}$; Williams \cite{williams07} showed that the face poset of $\Gr_{k,n}^{\ge 0}$ is EL-shellable, and Galashin, Karp, and Lam \cite{galashin_karp_lam22b,galashin_karp_lam22} showed that $\Gr_{k,n}^{\ge 0}$ is a regular CW complex homeomorphic to a closed ball. See \cref{rem:related_work} for further discussion of related work. In the case we consider here, $m=1$, we make use of the explicit description of the face poset $R_{n,l}$ of a cell decomposition of $\mathcal{A}_{n,k,m}(Z)$. No such description is known as yet for general $m$. For work in this direction, see \cite{karp_williams_zhang20,even-zohar_lakrec_tessler} for the case $m=4$, and \cite{lukowski,bao_he,lukowski_parisi_williams} for the case $m=2$.

Another consequence of \cref{thm:shelling_intro} is that the poset $R_{n,l}$ has a nonnegative $h$-vector. In particular, by a result of Bj\"{o}rner \cite{bjorner80} and Stanley \cite{stanley72a}, $h_i$ equals the number of maximal chains of $R_{n,l}$ with exactly $i$ descents with respect to the EL-labeling of \cref{thm:shelling_intro} (see \cref{thm:labeling_to_hvector}). We give an alternative description of the $h$-vector using generating functions (see \cref{thm:FH}), which is explicit but non-positive.

We observe that when $l=0$, the poset $R_{n,l}$ is the Boolean algebra $\bool{n}$ (consisting of all subsets of $\{1, \dots, n\}$ ordered by containment) with the minimum removed. Maximal chains of $R_{n,0}$ correspond to permutations of $\{1, \dots, n\}$ with the usual notion of descent, and $h_i$ is the Eulerian number $\eulerian{n}{i}$ (see \cref{prop:eulerian_case}). Therefore the $h$-vector of $R_{n,l}$ provides a generalization of the Eulerian numbers.

Two further well-known properties of the Boolean algebra $\bool{n}$ are that its rank sizes form a {\itshape log-concave sequence} and that it is {\itshape strongly Sperner} (see e.g.\ \cite{engel97}). We show that $R_{n,l}$ also has these properties:
\begin{theorem}\label{thm:sperner_intro}
The poset $R_{n,l}$ is rank-log-concave. It also admits a normalized flow, and hence is strongly Sperner.
\end{theorem}

Finally, we consider a poset closely related to $R_{n,l}$, denoted $P_{n,l}$, introduced by Machacek \cite{machacek}. It consists of projective sign vectors of length $n$ with at most (rather than exactly) $l$ sign changes, under the relation \eqref{sign_vector_definition}. For example, $P_{3,1}$ is depicted in \cref{fig:P31}. The poset $P_{n,l}$ can be regarded as a quotient of the face poset of a certain simplicial complex $\mathcal{B}(l,n)$ studied by Klee and Novik \cite{klee_novik12}. Notice that $P_{n,l}$ also specializes to $\bool{n}$ when $l=0$. Machacek~\cite{machacek} showed that the order complex of $P_{n,l}$ is a manifold with boundary which is homotopy equivalent to $\mathbb{RP}^l$, and homeomorphic to $\mathbb{RP}^{n-1}$ when $l = n-1$. Although $\hat{P}_{n,l}$ is not shellable in general, Bergeron, Dermenjian, and Machacek~\cite{bergeron_dermenjian_machacek20} showed that when $l$ is even or $l = n-1$, the order complex of $P_{n,l}$ is partitionable. This is a weaker property which still implies that the $h$-vector is nonnegative, and they showed that the $h$-vector counts certain type-$D$ permutations with respect to type-$B$ descents.

We prove that $P_{n,l}$ is rank-log-concave (see \cref{thm:P_unimodal}), and we conjecture that it is Sperner (see \cref{conj:sperner}). We prove this conjecture when $l$ equals $0$, $1$, or $n-1$ by constructing a normalized flow (see \cref{prop:sperner_conj}).

The remainder of this paper is organized as follows. In \cref{sec:shelling} we give some background on poset topology and prove \cref{thm:shelling_intro} (see \cref{thm:EL}). In \cref{sec:face_numbers} we consider the $f$-vector and $h$-vector of $R_{n,l}$. In \cref{sec:sperner} we give background on unimodality, log-concavity, and the Sperner property, and prove \cref{thm:sperner_intro}. In \cref{sec:P} we consider the poset $P_{n,l}$.

\begin{figure}[t]
\begin{center}
    \begin{tikzpicture}[xscale=1.2]
            \node (pm0) at (-2,0) {$(+,-,0)$};
            \node (p0m) at (0,0) {$(+,0,-)$};
            \node (0pm) at (2,0) {$(0,+,-)$};
            \node (pmm) at (-1,2) {$(+,-,-)$};
            \node (ppm) at (1,2) {$(+,+,-)$};
     
            \draw (pm0.north) -- (pmm.south);
            \draw (p0m.north) -- (ppm.south);
            \draw (p0m.north) -- (pmm.south);
            \draw (0pm.north) -- (ppm.south);
    \end{tikzpicture}\hspace*{36pt}
    \begin{tikzpicture}[xscale=1.2]
            \node (pm0) at (-2,0) {$(\{1\},\{2\})$};
            \node (p0m) at (0,0) {$(\{1\},\{3\})$};
            \node (0pm) at (2,0) {$(\{2\},\{3\})$};
            \node (pmm) at (-1,2) {$(\{1\},\{2,3\})$};
            \node (ppm) at (1,2) {$(\{1,2\},\{3\})$};
     
            \draw (pm0.north) -- (pmm.south);
            \draw (p0m.north) -- (ppm.south);
            \draw (p0m.north) -- (pmm.south);
            \draw (0pm.north) -- (ppm.south);
    \end{tikzpicture}
\end{center}
    \caption{The Hasse diagram of the poset $R_{3,1}$, with elements labeled as sign vectors (left) and as tuples of sets (right).}
    \label{fig:R31}
\end{figure}
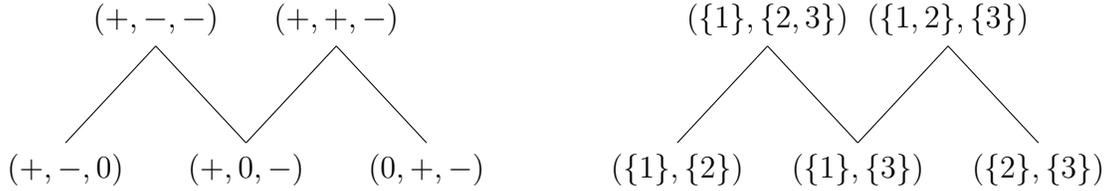

\subsection*{Acknowledgments}
We thank Isabella Novik and Bruce Sagan for helpful comments, and anonymous reviewers for their valuable feedback.

\section{EL-labeling}\label{sec:shelling}

\subsection{Notation and background}\label{sec:notation}

We let $\mathbb{N}$ denote $\{0, 1, 2, \dots\}$. For $n\in\mathbb{N}$ we define $[n] := \{1, 2, \dots, n\}$, and for $0 \le k \le n$ we let $\binom{[n]}{k}$ denote the set of $k$-element subsets of $[n]$. We let $\sym{n}$ denote the symmetric group of all permutations of $[n]$.

We assume the reader has some familiarity with posets; we refer to \cite{stanley12,wachs07} for further background. We use $\lessdot$ to denote cover relations in a poset, i.e., $x \lessdot y$ if and only if $x < y$ and there does not exist $z$ such that $x < z < y$.
\begin{definition}\label{def:graded}
Let $P$ be a finite poset. We say that $P$ is {\itshape graded} (or {\itshape pure}) if every maximal chain has the same length $d$, which we call the {\itshape rank} of $P$.
\end{definition}

\begin{definition}\label{def:hat}
Let $P$ be a poset. We define the {\itshape bounded extension} as the poset $\hat{P}$ obtained from $P$ by adjoining a new minimum $\hat{0}$ and a new maximum $\hat{1}$.
\end{definition}

We now recall the definition of an EL-labeling, due to Bj\"{o}rner \cite[Definition 2.1]{bjorner80}. We slightly modify the original definition, following Wachs \cite{wachs07}; see \cite[Remark 3.2.5]{wachs07} for further discussion.
\begin{definition}[{\cite[Definition 3.2.1]{wachs07}}]\label{def:EL-labeling}
Let $P$ be a finite graded poset. An {\itshape edge labeling} of $P$ is a function $\lambda$ from the set of edges of the Hasse diagram of $P$ (i.e.\ the cover relations of $P$) to a poset $(\Lambda,\preceq)$. An {\itshape increasing chain} is a saturated chain $x_0 \lessdot x_1 \lessdot \cdots \lessdot x_r$ in $P$ whose edge labels strictly increase in $\Lambda$:
$$
\lambda(x_0\lessdot x_1) \prec \lambda(x_1\lessdot x_2) \prec \cdots \prec \lambda(x_{r-1}\lessdot x_r).
$$
We call $\lambda$ an {\itshape EL-labeling} of $P$ if the following properties hold for every closed interval $[x,y]$ in $P$:
\begin{enumerate}[label={(EL\arabic*)},leftmargin=48pt,itemsep=2pt]
\item\label{EL1} there exists a unique increasing maximal chain $C_0$ in $[x,y]$; and
\item\label{EL2} if $x \lessdot z \le y$ such that $z\neq x_1$, where $x\lessdot x_1$ is the first edge of $C_0$, then $\lambda(x\lessdot x_1) \prec \lambda(x\lessdot z)$.\footnote{One may replace \ref{EL2} by the condition that $C_0$ is lexicographically minimal among all maximal chains of $[x,y]$; see \cite[Proposition 2.5]{bjorner80}.}
\end{enumerate}

\end{definition}

Bj\"{o}rner showed that a finite graded poset with an EL-labeling is shellable \cite[Theorem 2.3]{bjorner80}.

\subsection{Edge labeling}\label{sec:EL}

We now study the bounded extension $\hat{R}_{n,l}$ of $R_{n,l}$. Recall that $R_{n,l}$ consists of all projective sign vectors of length $n$ with exactly $l$ sign changes, under the relation \eqref{sign_vector_definition}. We begin by giving an alternative definition of $R_{n,l}$.
\begin{definition}\label{def:tuples}
Let $0 \le l < n$. We may equivalently define $R_{n,l}$ as follows. Its elements are $(l+1)$-tuples $(A_1, \dots, A_{l+1})$ of nonempty subsets of $[n]$ (called {\itshape blocks}) such that $\max(A_i) < \min(A_{i+1})$ for all $i \in [l]$. The order relation on $(l+1)$-tuples is given by component-wise containment:
$$
(A_1, \dots, A_{l+1}) \le (B_1, \dots, B_{l+1}) \quad \iff \quad A_i \subseteq B_i \text{ for } i\in [l+1].
$$
We may verify that this is equivalent to the definition of $R_{n,l}$ from \eqref{sign_vector_definition}, where an $(l+1)$-tuple of subsets records the positions of the consecutive runs of $+$'s and $-$'s in a sign vector. That is, $(A_1, \dots, A_{l+1})$ corresponds to the sign vector $(v_1, \dots, v_n)$ such that for $1 \le i \le n$,
$$
v_i = \begin{cases}
(-1)^{j-1}, & \text{ if $i\in A_j$ for some $j\in [l+1]$}; \\
0, & \text{ if $i\notin A_1 \cup \cdots \cup A_{l+1}$}.
\end{cases}
$$
For example, in $R_{9,2}$, the tuple of sets $(\{1,2,4\},\{6,8\},\{9\})$ corresponds to the sign vector $(+,+,0,+,0,-,0,-,+)$. Also see \cref{fig:R31}.
\end{definition}
We observe that the bounded extension $\hat{R}_{n,l}$ of $R_{n,l}$ is graded. Explicitly, the minimum $\hat{0}$ has rank $0$, the maximum $\hat{1}$ has rank $n-l+1$, and $(A_1, \dots, A_{l+1})$ has rank $|A_1| + \cdots + |A_l| - l$.

We now divide the cover relations of $R_{n,l}$ into two types; see \cref{rem:type_motivation} for motivation.
\begin{definition}\label{lem:covers}
Let $0 \le l < n$, and let $x = (A_1, \dots, A_{l+1})\in R_{n,l}$. Note that the elements of $R_{n,l}$ which cover $x$ are precisely those that can be obtained from it by adding some element $a\in [n]\setminus(A_1\cup\cdots\cup A_{l+1})$ to the $i$th block $A_i$, where $i\in [l+1]$ such that
$$
\max(A_{i-1}) < a < \min(A_{i+1}).
$$
(We take the inequality above to be $a < \min(A_2)$ when $i=1$, and $\max(A_l) < a$ when $i=l+1$.) We say that such a cover relation is of {\itshape type $\alpha$} if $a < \max(A_i)$, and of {\itshape type $\beta$} if $a > \max(A_i)$.
\end{definition}

\begin{definition}\label{def:labels}
Let $0 \le l < n$. We define a total order $(\Lambda_{n,l},\preceq)$ on the disjoint union of $\{\alpha,\beta\}\times [l+1]\times [n+1]$ and $\binom{[n]}{l+1}$, as follows (where $\ast$ denotes an arbitrary number):
\begin{itemize}[itemsep=2pt]
\item $(\alpha,\ast,\ast) \prec I \prec (\beta,\ast,\ast)$ for all $I\in\binom{[n]}{l+1}$;
\item $(\alpha,i,\ast) \prec (\alpha,j,\ast)$ and $(\beta,i,\ast) \succ (\beta,j,\ast)$ for all $i<j$ in $[l+1]$;
\item $(\alpha,i,a) \prec (\alpha,i,b)$ and $(\beta,i,a) \prec (\beta,i,b)$ for all $i\in [l+1]$ and $a<b$ in $[n+1]$; and
\item $\binom{[n]}{l+1}$ is ordered lexicographically: $\{1, \dots, l+1\} \prec \cdots \prec \{n-l, \dots, n\}$.
\end{itemize}
We define an edge labeling on $\hat{R}_{n,l}$, with label set $\Lambda_{n,l}$, as follows.
\begin{enumerate}[label=(\roman*), leftmargin=*, itemsep=2pt]
\item\label{labels1} We label the edge $\hat{0} \lessdot (\{a_1\}, \dots, \{a_{l+1}\})$ by $\{a_1, \dots, a_{l+1}\}\in\binom{[n]}{l+1}$.
\item\label{labels2} Let $\hat{0} < x \lessdot y < \hat{1}$. Then as in \cref{lem:covers}, $y$ is obtained from $x$ by adding some element $a$ to the $i$th block of $x$, in a cover relation of type $\gamma$ (where $\gamma\in\{\alpha,\beta\}$). We label the edge $x\lessdot y$ by $(\gamma,i,a)$.
\item\label{labels3} We label the edge $x\lessdot\hat{1}$ by $(\beta,l+1,n+1)$.
\end{enumerate}
For example, see \cref{fig:R31_labeled}.
\end{definition}

\begin{remark}\label{rem:type_motivation}
We were led to the construction in \cref{def:labels} in part so that the following property holds (though we will not end up using it). Let $x\in\hat{R}_{n,l}$ such that $x$ is not covered by $\hat{1}$, and let $y_1, \dots, y_r$ be the elements of $\hat{R}_{n,l}$ which cover $x$, ordered so that the labels of $x\lessdot y_1, \dots, x\lessdot y_r$ are increasing in $(\Lambda_{n,l},\preceq)$. Then $y_1, \dots, y_r$ are in increasing order in the lexicographic order on $(l+1)$-tuples. For example, see \cref{fig:R92}. In fact, one can show that ordering the atoms of $[x,\hat{1}]$ lexicographically for all $x\in\hat{R}_{n,l}$ defines a {\itshape recursive atom ordering} of $\hat{R}_{n,l}$ (see e.g.\ \cite[Section 4.2]{wachs07}), where the order of the atoms of $[x,\hat{1}]$ does not depend on a choice of maximal chain of $[\hat{0},x]$. Li \cite[Lemma 1.1]{li21} showed that any finite, bounded, and graded poset admitting such a recursive atom ordering is EL-shellable, so this provides an alternative way to prove \cref{thm:shelling_intro}. We omit the proof of this fact, and instead find it simplest to work only with the edge labeling in \cref{def:labels}.
\end{remark}

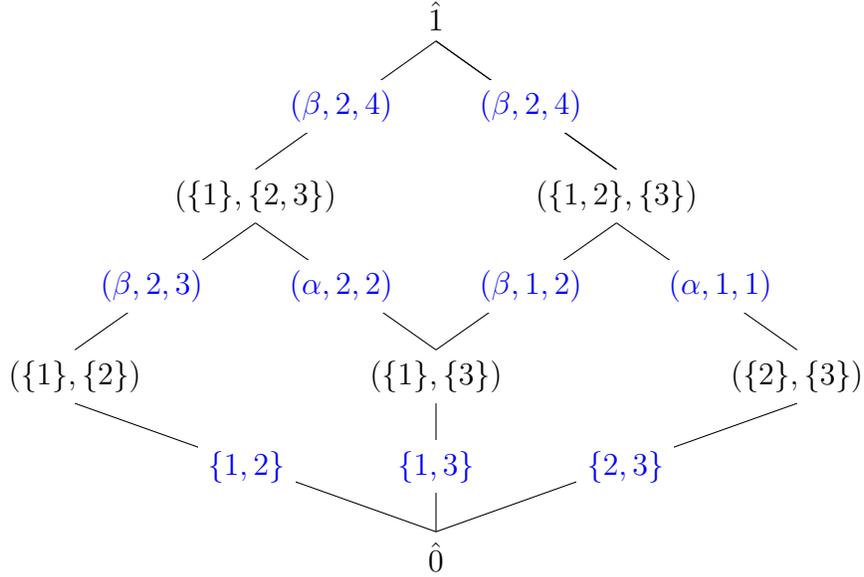
\begin{figure}[t]
\begin{center}
    \begin{tikzpicture}[xscale=2.4,yscale=1.2]
            \node (min) at (0,-2) {$\hat{0}$};
            \node (pm0) at (-2,0) {$(\{1\},\{2\})$};
            \node (p0m) at (0,0) {$(\{1\},\{3\})$};
            \node (0pm) at (2,0) {$(\{2\},\{3\})$};
            \node (pmm) at (-1,2) {$(\{1\},\{2,3\})$};
            \node (ppm) at (1,2) {$(\{1,2\},\{3\})$};
            \node (max) at (0,4) {$\hat{1}$};
     
            \draw (min.north) -- (pm0.south);
            \draw (min.north) -- (p0m.south);
            \draw (min.north) -- (0pm.south);
            \draw (pm0.north) -- (pmm.south);
            \draw (p0m.north) -- (ppm.south);
            \draw (p0m.north) -- (pmm.south);
            \draw (0pm.north) -- (ppm.south);
            \draw (ppm.north) -- (max.south);
            \draw (pmm.north) -- (max.south);
            \draw (ppm.north) -- (max.south);
            
            \node at (-1.05,-1) [color=blue,fill=white] {$\{1,2\}$};
            \node at (0,-1) [color=blue,fill=white] {$\{1,3\}$};
            \node at (1.05,-1) [color=blue,fill=white] {$\{2,3\}$};
            \node at (-1.575,1) [color=blue,fill=white] {$(\beta,2,3)$};
            \node at (-0.525,1) [color=blue,fill=white] {$(\alpha,2,2)$};
            \node at (0.525,1) [color=blue,fill=white] {$(\beta,1,2)$};
            \node at (1.575,1) [color=blue,fill=white] {$(\alpha,1,1)$};
            \node at (-0.525,3) [color=blue,fill=white] {$(\beta,2,4)$};
            \node at (0.525,3) [color=blue,fill=white] {$(\beta,2,4)$};
    \end{tikzpicture}
\end{center}
    \caption{The edge labeling of $\hat{R}_{3,1}$ in \cref{def:labels}.}
    \label{fig:R31_labeled}
\end{figure}

\begin{figure}[t]
\begin{center}
    \begin{tikzpicture}[xscale=1.2,yscale=1.5]
            \pgfmathsetmacro{\da}{0.35};
            
            \node (min) at (0,-2) {$(24,6,8)$};
            \node (x1) at (-6,0) {$(124,6,8)$};
            \node (x2) at (-4,0) {$(234,6,8)$};
            \node (x3) at (-2,0) {$(24,56,8)$};
            \node (x4) at (0,0) {$(24,6,78)$};
            \node (x5) at (2,0) {$(24,6,89)$};
            \node (x6) at (4,0) {$(24,67,8)$};
            \node (x7) at (6,0) {$(245,6,8)$};

            \draw (min.north) -- (x1.south);
            \draw (min.north) -- (x2.south);
            \draw (min.north) -- (x3.south);
            \draw (min.north) -- (x4.south);
            \draw (min.north) -- (x5.south);
            \draw (min.north) -- (x6.south);
            \draw (min.north) -- (x7.south);
            
            \node at (-6.3+6.3*\da,-2*\da) [color=blue,fill=white] {$(\alpha,1,1)$};
            \node at (-4.2+4.2*\da,-2*\da) [color=blue,fill=white] {$(\alpha,1,3)$};
            \node at (-2.1+2.1*\da,-2*\da) [color=blue,fill=white] {$(\alpha,2,5)$};
            \node at (0,-2*\da) [color=blue,fill=white] {$(\alpha,3,7)$};
            \node at (2.1-2.1*\da,-2*\da) [color=blue,fill=white] {$(\beta,3,9)$};
            \node at (4.2-4.2*\da,-2*\da) [color=blue,fill=white] {$(\beta,2,7)$};
            \node at (6.3-6.3*\da,-2*\da) [color=blue,fill=white] {$(\beta,1,5)$};
    \end{tikzpicture}
\end{center}
\caption{The element $(\{2,4\},\{6\},\{8\})\in\hat{R}_{9,2}$, and the elements covering it ordered by increasing edge label (equivalently, ordered lexicographically as $(l+1)$-tuples).}
    \label{fig:R92}
\end{figure}
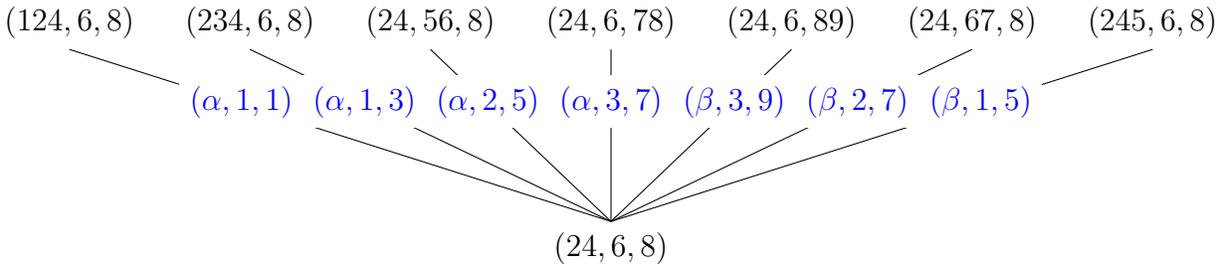

\begin{theorem}\label{thm:EL}
The edge labeling of $\hat{R}_{n,l}$ in \cref{def:labels} is an EL-labeling.
\end{theorem}

\begin{proof}
We must verify that \ref{EL1} and \ref{EL2} hold for every closed interval $[x,y]$ in $\hat{R}_{n,l}$. We consider four cases, depending on whether $x=\hat{0}$ and $y=\hat{1}$. When $x\neq\hat{0}$ we write $x = (A_1, \dots, A_{l+1})$, and when $y\neq\hat{1}$ we write $y = (B_1, \dots, B_{l+1})$. In each case, we explicitly describe the unique maximal chain of $[x,y]$, thereby proving \ref{EL1}. It will then be apparent from the form of this maximal chain that \ref{EL2} holds.

{\itshape Case 1: $x\neq\hat{0}$, $y\neq\hat{1}$.} The maximal chains of $[x,y]$ are obtained by adding, in some order, all the elements of $B_i\setminus A_i$ to the $i$th block (for $i\in [l+1]$). The unique increasing chain is given by adding these elements in the following order:
\begin{itemize}[itemsep=2pt]
\item for $i = 1, \dots, l+1$, we add the elements of $B_i\setminus A_i$ which are less than $\max(A_i)$ to the $i$th block, in increasing order (in cover relations of type $\alpha$);
\item for $i = l+1, \dots, 1$, we add the elements of $B_i\setminus A_i$ which are greater than $\max(A_i)$ to the $i$th block, in increasing order (in cover relations of type $\beta$).
\end{itemize}
We see that \ref{EL2} holds.

{\itshape Case 2: $x=\hat{0}$, $y\neq\hat{1}$.} The first edge of any maximal chain of $[\hat{0},y]$ is labeled by an element of $\binom{[n]}{l+1}$, and so if it is increasing, after the first edge it must pass through edges only of type $\beta$. Therefore the unique increasing maximal chain of $[\hat{0},y]$ begins with the edge $\hat{0}\lessdot (\{b_1\}, \dots, \{b_{l+1}\})$, where $b_i := \min(B_i)$ for $i\in [l+1]$ (whence \ref{EL2} is satisfied), and after that follows the unique increasing chain from $(\{b_1\}, \dots, \{b_{l+1}\})$ to $y$, as in Case 1.

{\itshape Case 3: $x\neq\hat{0}$, $y=\hat{1}$.} The last edge of any maximal chain of $[x,\hat{1}]$ is labeled by $(\beta,l+1,n+1)$, and so if it is increasing, before the final edge it must pass through edges only of type $\alpha$ or with a label $(\beta,l+1,\ast)$. Therefore the unique increasing maximal chain of $[x,\hat{1}]$ ends with the edge $(C_1, \dots, C_{l+1})\lessdot\hat{1}$, where
\begin{multline*}
C_1 := \{1, 2, \dots, \max(A_1)\}, C_2 := \{\max(A_1)+1, \max(A_1)+2, \dots, \max(A_2)\},\dots,\\
C_{l+1} := \{\max(A_l)+1, \max(A_l)+2, \dots, n\},
\end{multline*}
and before that follows the unique increasing chain from $x$ to $(C_1, \dots, C_{l+1})$, as in Case 1. We see that \ref{EL2} holds.

{\itshape Case 4: $x=\hat{0}$, $y=\hat{1}$.} Reasoning as in Cases 2 and 3, the unique increasing maximal chain of $[\hat{0},\hat{1}]$ begins with the edge $\hat{0}\lessdot (\{1\}, \dots, \{l+1\})$, ends with the edge $(\{1\}, \dots, \{l\}, \{l+1, \dots, n\})\lessdot\hat{1}$, and in between follows the unique increasing chain as in Case 1. As in Case 2, \ref{EL2} holds.
\end{proof}

\begin{remark}\label{rem:lattice}
There are results in the literature which imply that various special families of posets are shellable. However, as far as we know, $\hat{R}_{n,l}$ is not contained in such a family. For example, Provan and Billera \cite[Section 3.4.2]{provan_billera80} showed that all {\itshape distributive lattices} (cf.\ \cite[Section 3.4]{stanley12}) are shellable. While $\hat{R}_{n,l}$ is a lattice, it is not distributive unless $l=0$ (in which case $R_{n,l}$ is the Boolean algebra $\bool{n}$ with the minimum removed) or $l=n-1$ (in which case $R_{n,l}$ has a single element). For example, one can see from \cref{fig:R31_labeled} that $\hat{R}_{3,1}$ is not distributive. Also, Bj\"{o}rner \cite[Theorem 3.1]{bjorner80} showed that all {\itshape semimodular lattices} (cf.\ \cite[Section 3.3]{stanley12}) are shellable. However, $\hat{R}_{n,l}$ is not upper-semimodular unless $l=0$ or $l=n-1$, and $\hat{R}_{n,l}$ is not lower-semimodular unless $l=0$, $l=n-1$, or $(n,l) = (3,1)$. For example, one can see from \cref{fig:R31_labeled} that $\hat{R}_{3,1}$ is not upper-semimodular. We omit the proofs of these claims.
\end{remark}

\begin{remark}\label{rem:related_work}
Recall that $R_{n,l}$ is the face poset of the amplituhedron $\mathcal{A}_{n,k,m}(Z)$ when $m=1$. Another interesting special case of $\mathcal{A}_{n,k,m}(Z)$ is $n=k+m$, whence it becomes the totally nonnegative Grassmannian $\Gr_{k,n}^{\ge 0}$. Williams \cite{williams07} and Bao and He \cite[Theorem 4.1]{bao_he21} showed that the face poset of $\Gr_{k,n}^{\ge 0}$ with a minimum $\hat{0}$ adjoined is EL-shellable, and Knutson, Lam, and Speyer \cite[Section 3.5]{knutson_lam_speyer13} showed that the face poset (without $\hat{0}$ adjoined) is dual EL-shellable. We point out that $R_{n,l}$ with $\hat{0}$ adjoined (but not $\hat{1}$) is an induced subposet of the face poset of $\Gr_{k,n}^{\ge 0}$ with $\hat{0}$ adjoined \cite[Theorem 5.17]{karp_williams19}, and so it is EL-shellable by \cite{williams07,bao_he21}. Therefore the main difficulty in proving \cref{thm:shelling_intro} is in dealing with the adjoined maximum $\hat{1}$. Our EL-labeling of $\hat{R}_{n,l}$ does not use the labelings of \cite{williams07,knutson_lam_speyer13,bao_he21}, and it is not clear to us how our labeling is related to these. We plan to study this further in future work.
\end{remark}

\section{\texorpdfstring{$f$}{f}-vector and \texorpdfstring{$h$}{h}-vector}\label{sec:face_numbers}

\noindent In this section we examine the $f$-vector and $h$-vector of $R_{n,l}$, as well as their refinements by ranks, namely the flag $f$-vector and flag $h$-vector. We give a combinatorial interpretation for the $h$-vector in terms of the EL-labeling of \cref{sec:EL}, and also prove explicit formulas for the $f$-vector and $h$-vector.

\subsection{Background}\label{sec:face_numbers_background}

We refer to \cite{stanley96,stanley12} for background on the $f$-vector and $h$-vector.
\begin{definition}[{\cite[Section 3.13]{stanley12}}]\label{def:hvector}
Let $P$ be a finite graded poset of rank $d-1$, with ranks labeled from $1$ to $d$. For $S\subseteq [d]$, we let $\alpha_S$ be the number of chains of $P$ supported exactly at the ranks in $S$; we call $\alpha$ the {\itshape flag $f$-vector} of $P$. We also define the {\itshape flag $h$-vector} $\beta$ of $P$ by
$$
\beta_S := \sum_{T\subseteq S}(-1)^{|S\setminus T|}\alpha_T, \quad \text{ or equivalently}, \quad \alpha_S =: \sum_{T\subseteq S}\beta_T \quad (S\subseteq [d]).
$$
Alternatively, let $P_S$ denote the induced subposet of $P$ consisting of all elements whose rank lies in $S$. Then $\alpha_S$ is the number of maximal chains of $P_S$, and $(-1)^{|S|+1}\beta_S$ is the M\"{o}bius invariant $\mu(\hat{P}_S)$ of the bounded extension of $P_S$.

We define the $f$-vector $(f_{-1}, f_0, \dots, f_{d-1})$ and $h$-vector $(h_0, \dots, h_d)$ of $P$ by
$$
f_{i-1} := \sum_{S\in\binom{[d]}{i}}\alpha_S \quad \text{ and } \quad h_i := \sum_{S\in\binom{[d]}{i}}\beta_S \quad (0 \le i \le d).
$$
Defining the generating functions\footnote{Our $F(t)$ and $H(t)$ are the reverses of the generating functions in \cite{stanley12}.}
$$
F(t) := \sum_{i=0}^df_{i-1}t^i \quad \text{ and } \quad H(t) := \sum_{i=0}^dh_it^i,
$$
the $f$-vector and $h$-vector are related by the equation
\begin{align}\label{eq:FtoH}
H(t) = (1-t)^dF(\textstyle\frac{t}{1-t}).
\end{align}

\end{definition}

\begin{remark}\label{rem:hvector_bounded}
Let $P$ be a finite graded poset of rank $d-1$, with ranks labeled from $1$ to $d$. Let $\hat{P}$ denote the bounded extension of $P$, with ranks labeled from $0$ to $d+1$. Then for all $S\subseteq\{0, \dots, d+1\}$, we have \cite[p.\ 294]{stanley12}
$$
\alpha_S(\hat{P}) = \alpha_{S\setminus\{0,d+1\}}(P) \quad \text{ and } \quad \beta_S(\hat{P}) = \begin{cases}
0, & \text{ if $0\in S$ or $d+1\in S$}; \\
\beta_S(P), & \text{ otherwise}.
\end{cases}
$$
In particular, $P$ and $\hat{P}$ have the same (flag) $h$-vector, and the (flag) $f$-vector of $\hat{P}$ is easily determined from $P$. Therefore enumerative results for $R_{n,l}$ apply as well to $\hat{R}_{n,l}$, and vice-versa. Keeping this connection in mind, we will label the ranks of $R_{n,l}$ from $1$ to $n-l$ (rather than from $0$ to $n-l-1$).
\end{remark}

\begin{example}\label{eg:hvector}
Consider the poset $R_{3,1}$, shown in \cref{fig:R31}. Then $d=2$, and
\begin{align*}
&\alpha_\varnothing = 1, &&\alpha_{\{1\}} = 3, &&\alpha_{\{2\}} = 2, &&\alpha_{\{1,2\}} = 4, &&(f_{-1},f_0,f_1) = (1,5,4), &&F(t) = 1+5t+4t^2; \\
&\beta_\varnothing = 1, &&\beta_{\{1\}} = 2, &&\beta_{\{2\}} = 1, &&\beta_{\{1,2\}} = 0, &&(h_0,h_1,h_2) = (1,3,0), &&H(t) = 1+3t.
\end{align*}
\end{example}

\subsection{Combinatorial interpretations}\label{sec:face_numbers_interpretations}

Bj\"{o}rner \cite[Theorem 2.7]{bjorner80}, based on work of Stanley \cite[Theorem 1.2]{stanley72a}, gave a combinatorial interpretation for the flag $h$-vector of any poset with an edge labeling satisfying \ref{EL1}. We state it here in the special case of $\hat{R}_{n,l}$, with the edge labeling defined in \cref{def:labels}.
\begin{definition}\label{def:chain_descent}
Given a maximal chain $\hat{0} = x_0 \lessdot x_1 \lessdot \cdots \lessdot x_{n-l} \lessdot x_{n-l+1} = \hat{1}$ of $\hat{R}_{n,l}$, we say that $i\in [n-l]$ is a {\itshape descent} of $C$ when $\lambda(x_{i-1} \lessdot x_i) \succ \lambda(x_i \lessdot x_{i+1})$.\footnote{For edge labelings of general posets, one should replace `$\succ$' with `$\nprec$' in the definition. There is no difference for our edge labeling of $\hat{R}_{n,l}$, since the label set $\Lambda_{n,l}$ is totally ordered and no label is repeated in any maximal chain.}
\end{definition}

\begin{figure}[t]
\begin{center}
\begin{tabular}{cccc}
    \hspace*{12pt}\begin{tikzpicture}[yscale=1.0]
            \node at (0,-3) {$S = \varnothing$};
            \node (p0) at (0,-2) {$\hat{0}$};
            \node (p1) at (0,0) {$(\{1\},\{2\})$};      
            \node (p2) at (0,2) {$(\{1\},\{2,3\})$};          
            \node (p3) at (0,4) {$\hat{1}$};           
     
            \draw (p0.north) -- (p1.south);
            \draw (p1.north) -- (p2.south);
            \draw (p2.north) -- (p3.south);
            
            \node at (0,-1) [color=blue,fill=white] {$\{1,2\}$};
            \node at (0,1) [color=blue,fill=white] {$(\beta,2,3)$};
            \node at (0,3) [color=blue,fill=white] {$(\beta,2,4)$};
    \end{tikzpicture}\hspace*{12pt}&
    \hspace*{12pt}\begin{tikzpicture}[yscale=1.0]
            \node at (0,-3) {$S = \{1\}$};
            \node (p0) at (0,-2) {$\hat{0}$};
            \node (p1) at (0,0) {$(\{1\},\{3\})$};      
            \node (p2) at (0,2) {$(\{1\},\{2,3\})$};          
            \node (p3) at (0,4) {$\hat{1}$};           
     
            \draw (p0.north) -- (p1.south);
            \draw (p1.north) -- (p2.south);
            \draw (p2.north) -- (p3.south);
            
            \node at (0,-1) [color=blue,fill=white] {$\{1,3\}$};
            \node at (0,1) [color=blue,fill=white] {$(\alpha,2,2)$};
            \node at (0,3) [color=blue,fill=white] {$(\beta,2,4)$};
    \end{tikzpicture}\hspace*{12pt}&
    \hspace*{12pt}\begin{tikzpicture}[yscale=1.0]
            \node at (0,-3) {$S = \{2\}$};
            \node (p0) at (0,-2) {$\hat{0}$};
            \node (p1) at (0,0) {$(\{1\},\{3\})$};      
            \node (p2) at (0,2) {$(\{1,2\},\{3\})$};          
            \node (p3) at (0,4) {$\hat{1}$};           
     
            \draw (p0.north) -- (p1.south);
            \draw (p1.north) -- (p2.south);
            \draw (p2.north) -- (p3.south);
            
            \node at (0,-1) [color=blue,fill=white] {$\{1,3\}$};
            \node at (0,1) [color=blue,fill=white] {$(\beta,1,2)$};
            \node at (0,3) [color=blue,fill=white] {$(\beta,2,4)$};
    \end{tikzpicture}\hspace*{12pt}&
    \hspace*{12pt}\begin{tikzpicture}[yscale=1.0]
            \node at (0,-3) {$S = \{1\}$};
            \node (p0) at (0,-2) {$\hat{0}$};
            \node (p1) at (0,0) {$(\{2\},\{3\})$};      
            \node (p2) at (0,2) {$(\{1,2\},\{3\})$};          
            \node (p3) at (0,4) {$\hat{1}$};           
     
            \draw (p0.north) -- (p1.south);
            \draw (p1.north) -- (p2.south);
            \draw (p2.north) -- (p3.south);
            
            \node at (0,-1) [color=blue,fill=white] {$\{2,3\}$};
            \node at (0,1) [color=blue,fill=white] {$(\alpha,1,1)$};
            \node at (0,3) [color=blue,fill=white] {$(\beta,2,4)$};
    \end{tikzpicture}\hspace*{12pt}
\end{tabular}
\end{center}
    \caption{The maximal chains of $\hat{R}_{3,1}$ and their descent sets $S$.}
    \label{fig:R31_chains}
\end{figure}
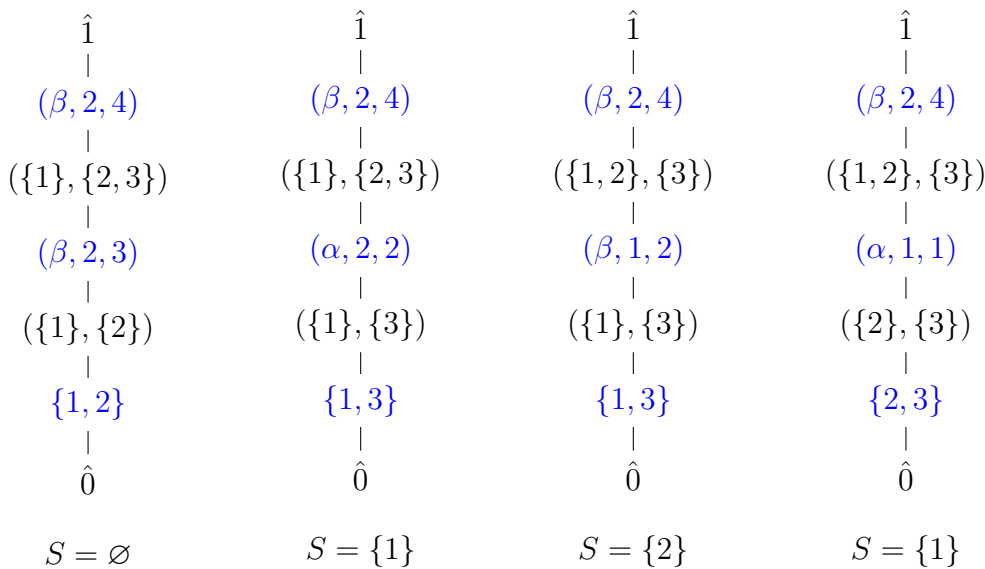

\begin{theorem}[{Bj\"{o}rner and Stanley; cf.\ \cite[Theorem 3.14.2]{stanley12}\footnote{Our conventions differ slightly from those in \cite{stanley12}, since in \ref{EL1} we require edge labels to strictly (rather than weakly) increase. Nevertheless, the result \cite[Theorem 3.14.2]{stanley12} and its proof transfer easily to our setting.}}]\label{thm:labeling_to_hvector}
Recall the edge labeling of $\hat{R}_{n,l}$ in \cref{def:labels}. For all $S\subseteq [n-l]$, we have that $\beta_S$ equals the number of maximal chains of $\hat{R}_{n,l}$ with descent set $S$. Thus for all $0 \le d \le n-l$, we have that $h_d$ equals the number of maximal chains of $\hat{R}_{n,l}$ with exactly $d$ descents.
\end{theorem}

\begin{example}\label{eg:labeling_to_hvector}
The maximal chains of $\hat{R}_{3,1}$ and their descent sets are shown in \cref{fig:R31_chains}. According to \cref{thm:labeling_to_hvector}, we have $\beta_\varnothing = 1$, $\beta_{\{1\}} = 2$, $\beta_{\{2\}} = 1$, and $\beta_{\{1,2\}} = 0$, consistent with \cref{eg:hvector}.
\end{example}

We also have the following explicit description of all the maximal chains of $R_{n,l}$ (and hence also $\hat{R}_{n,l}$):
\begin{proposition}\label{prop:maximalchains}
The number of maximal chains of $R_{n,l}$ is $f_{n-l-1} = \binom{n+l}{2l+1}(n-l-1)!$. Explicitly, given $A\in\binom{[n+l]}{2l+1}$ and a permutation $\pi\in\sym{n-l-1}$, we associate a maximal chain $C(A,\pi)$ of $R_{n,l}$ as follows:
\begin{itemize}
\item writing $A = \{c_1 < \cdots < c_{2l+1}\}$, we set $a_i := c_{2i-1} - i + 1$ for $1 \le i \le l+1$ and $b_i := c_{2i} - i$ for $1 \le i \le l$;
\item we take $C(A,\pi)$ to have minimal element $x := (\{a_1\}, \dots, \{a_{l+1}\})$ and maximal element $y := (\{1, \dots, b_1\}, \{b_1 + 1, \dots, b_2\}, \dots, \{b_l+1, \dots, n\})$;
\item writing $[n]\setminus\{a_1, \dots, a_{l+1}\} = \{a'_1, \dots, a'_{n-l-1}\}$ (in increasing order), $C(A,\pi)$ is given by adding the elements $a'_{\pi(1)}, \dots, a'_{\pi(n-l-1)}$ to $x$ (in that order), each to the appropriate block (determined by $y$).
\end{itemize}

\end{proposition}

\begin{proof}
We can verify that the map $(A,\pi) \mapsto C(A,\pi)$ gives a bijection from $\binom{[n+l]}{2l+1}\times\sym{n-l-1}$ to the set of maximal chains of $R_{n,l}$.
\end{proof}

For example, the maximal chains in \cref{fig:R31_chains} are (from left to right) $C(\{1,2,3\},1)$, $C(\{1,2,4\},1)$, $C(\{1,3,4\},1)$, and $C(\{2,3,4\},1)$.

While \cref{prop:maximalchains} gives a simple description of the maximal chains of $\hat{R}_{n,l}$, we are not able in general to translate \cref{def:labels} into a simple description of the descents of $C(A,\pi)$ in terms of $A$ and $\pi$. However, in the special case $l=0$, we do have such a simple description: maximal chains correspond to permutations of $[n]$ with the usual notion of descent, as we now explain.
\begin{definition}\label{def:eulerian_numbers}
Given $\pi\in\sym{n}$, we say that $r\in [n-1]$ is a {\itshape descent} of $\pi$ if $\pi(r) > \pi(r+1)$. For $0 \le d \le n$, we define the {\itshape Eulerian number} $\eulerian{n}{d}$ as the number of permutations in $\sym{n}$ with exactly $d$ descents.
\end{definition}

For example, $\eulerian{3}{1} = 4$, corresponding to the permutations (in one-line notation) $132$, $213$, $231$, and $312$. We refer to \cite{petersen15} for further details about Eulerian numbers.
\begin{proposition}\label{prop:eulerian_case}
There is a bijection between maximal chains of $\hat{R}_{n,0}$ and permutations in $\sym{n}$ which preserves descent sets. In particular, by \cref{thm:labeling_to_hvector}, we have $h_d = \eulerian{n}{d}$ for $0 \le d \le n$.
\end{proposition}

\begin{proof}
The bijection sends the permutation $\pi\in\sym{n}$ to the maximal chain
$$
\hat{0} \lessdot (\{\pi(1)\}) \lessdot (\{\pi(1),\pi(2)\}) \lessdot \cdots \lessdot (\{\pi(1), \dots, \pi(n)\}) \lessdot \hat{1}.
$$
We can verify that the notions of descent in \cref{def:labels} and \cref{def:eulerian_numbers} agree.
\end{proof}

\subsection{Explicit formulas}\label{sec:face_numbers_formulas}

We now turn to giving explicit formulas for the $f$-vector and $h$-vector of $R_{n,l}$.
\begin{proposition}\label{prop:flagfvector}
The flag $f$-vector of $R_{n,l}$ is given by
$$
\alpha_S = \binom{l+r_1-1}{l}\binom{n}{l+r_d}\binom{2l+r_d}{r_d-r_1}\binom{r_d-r_1}{r_2-r_1, \dots, r_d-r_{d-1}}
$$
for all $S = \{r_1 < \cdots < r_d\}\subseteq [n-l]$.
\end{proposition}

\begin{proof}
We enumerate the chains $x_1 < \cdots < x_d$ of $R_{n,l}$ supported at ranks $r_1, \dots, r_d$ as follows. Write $x_1 = (A_1, \dots, A_{l+1})$ and $x_d = (B_1, \dots, B_{l+1})$. Since $|B_1 \cup \cdots \cup B_{l+1}| = l+r_d$, the number of ways to choose $B_1 \cup \cdots \cup B_{l+1}$ is
$$
\binom{n}{l+r_d}.
$$
After relabeling the set $[n]$, we may assume that $B_1 \cup \cdots \cup B_{l+1} = [l+r_d]$.

Let $s_i$ denote the size of $A_i$ (for $1 \le i \le l+1$), so that $s_i \ge 1$ and $s_1 + \cdots + s_{l+1} = l+r_1$. The number of ways to choose $s_1, \dots, s_{l+1}$ is
$$
\binom{l+r_1-1}{l}.
$$

For $1 \le i \le l+1$, write $A_i = \{a_{i,1} < \cdots < a_{i,s_i}\}$, and set $b_i := \max(B_i)$. Then the $a_{i,j}$'s and $b_i$'s are arbitrary elements of $[l+r_d]$ subject to
$$
a_{1,1} < \cdots < a_{1,s_1} \le b_1 < a_{2,1} < \cdots < a_{2,s_2} \le b_2 < \cdots \le b_l < a_{l+1,1} < \cdots < a_{l+1,s_{l+1}}.
$$
The number of ways to choose the $a_{i,j}$'s and $b_i$'s is
$$
\binom{2l+r_d}{2l+r_1} = \binom{2l+r_d}{r_d-r_1},
$$
at which point $x_1$ and $x_d$ are fixed.

Finally, the elements $x_2, \dots, x_{d-1}$ are determined by a set composition of $(B_1 \cup \cdots \cup B_{l+1})\setminus (A_1 \cup \cdots \cup A_{l+1})$ with blocks of respective sizes $r_2-r_1, \dots, r_d-r_{d-1}$. The number of choices is
\begin{gather*}
\binom{r_d-r_1}{r_2-r_1, \dots, r_d-r_{d-1}}.\qedhere
\end{gather*}

\end{proof}

We now use \cref{prop:flagfvector} to give a formula for the $f$-vector of $R_{n,l}$. The following formula allows us to simplify the resulting sum, at the cost of introducing minus signs.
\begin{lemma}[{\cite[(1.94a)]{stanley12}}]\label{lem:incexc}
Let $s\in\mathbb{N}$ and $d\in\mathbb{Z}_{>0}$. Then
\begin{gather*}
\sum_{\substack{i_1, \dots, i_d \ge 1, \\[1pt] i_1 + \cdots + i_d = s}}\binom{s}{i_1, \dots, i_d} = \sum_{i=0}^d(-1)^i\binom{d}{i}(d-i)^s.
\end{gather*}

\end{lemma}

\begin{proof}
This follows from \cite[(1.94a)]{stanley12}, since both sides equal $d!\hspace*{1pt}S(s,d)$, where $S(s,d)$ is a Stirling number of the second kind. Alternatively, we can prove this directly from the inclusion-exclusion principle \cite[Theorem 2.1.1]{stanley12}.
\end{proof}

\begin{corollary}\label{cor:fvectornonpositive}
Let $0 \le l < n$ and $0 \le d \le n-l-1$.
The number of chains of $R_{n,l}$ of length $d$ which begin at rank $r$ and end at rank $r+s$ equals
$$
\sum_{i=0}^d(-1)^i\binom{d}{i}\binom{l+r-1}{l}\binom{n}{l+r+s}\binom{2l+r+s}{s}(d-i)^s.
$$
Then $f_d$ is given by summing the quantity above over all $r\ge 1$ and $s\ge 0$ (or alternatively $s\ge d$).
\end{corollary}

\begin{proof}
This follows from \cref{prop:flagfvector}, using \cref{lem:incexc}.
\end{proof}

\begin{example}
Taking $d=0$ in \cref{cor:fvectornonpositive}, we obtain the number of elements of $R_{n,l}$:
\begin{gather*}
f_0 = \sum_{r=1}^{n-l}\binom{l+r-1}{l}\binom{n}{l+r}.
\end{gather*}

\end{example}

Finally, we use \cref{cor:fvectornonpositive} to obtain the generating functions for the $f$- and $h$-vectors:
\begin{theorem}\label{thm:FH}
The generating functions for the $f$- and $h$-vectors of $R_{n,l}$ are given by
$$
F(t) = 1 + \sum_{j,r,s\ge 0}\binom{l+r}{l}\binom{n}{l+r+s+1}\binom{2l+r+s+1}{s}j^s\left(\frac{t}{1+t}\right)^{j+1}
$$
and
\begin{gather*}
H(t) = (1-t)^{n-l}\left(1 + \sum_{j,r,s\ge 0}\binom{l+r}{l}\binom{n}{l+r+s+1}\binom{2l+r+s+1}{s}j^st^{j+1}\right).
\end{gather*}

\end{theorem}

We then obtain an explicit formula (albeit with negative signs) for $h_i$ by taking the coefficient of $t^i$ in $H(t)$.
\begin{proof}
By \cref{cor:fvectornonpositive} (replacing $r-1$ by $r$), and then writing $d=i+j$ and applying the negative binomial theorem, we obtain
\begin{align*}
F(t) &= 1 + \sum_{d\ge 0}\sum_{i,r,s\ge 0}(-1)^i\binom{d}{i}\binom{l+r}{l}\binom{n}{l+r+s+1}\binom{2l+r+s+1}{s}(d-i)^st^{d+1} \\[2pt]
&= 1 + \sum_{i,j,r,s\ge 0}(-1)^i\binom{i+j}{i}\binom{l+r}{l}\binom{n}{l+r+s+1}\binom{2l+r+s+1}{s}j^st^{i+j+1} \\[2pt]
&= 1 + \sum_{j,r,s\ge 0}\binom{l+r}{l}\binom{n}{l+r+s+1}\binom{2l+r+s+1}{s}j^st^{j+1}(1+t)^{-(j+1)}.
\end{align*}
This proves the first equation. The second equation follows by applying \eqref{eq:FtoH}. 
\end{proof}

\begin{example}\label{eg:hvector0}
Let us set $l=0$ in \cref{thm:FH} to obtain the generating function for the $h$-vector of $R_{n,0}$:
\begin{align*}
H(t) &= (1-t)^n\left(1 + \sum_{j,r,s\ge 0}\binom{n}{r+s+1}\binom{r+s+1}{s}j^st^{j+1}\right) \\
&= (1-t)^n\left(1 + \sum_{j,r,s\ge 0}\binom{n}{s}\binom{n-s}{r+1}j^st^{j+1}\right) \\
&= (1-t)^n\left(1 + \sum_{j,s\ge 0}\binom{n}{s}(2^{n-s}-1)j^st^{j+1}\right) \\
&= (1-t)^n\left(1 + \sum_{j\ge 0}\big((j+2)^n - (j+1)^n\big)t^{j+1}\right) \\
&= (1-t)^{n+1}\sum_{j\ge 0}(j+1)^nt^j,
\end{align*}
where we applied the binomial theorem twice. This yields a well-known generating function for the Eulerian numbers \cite[(1.10)]{petersen15}, in agreement with \cref{prop:eulerian_case}.
\end{example}

\begin{example}\label{eg:h1}
Let us use \cref{thm:FH} to find $h_1$ for $R_{n,l}$, by taking the coefficient of $t$ in $H(t)$:
$$
h_1 = l-n + \sum_{r\ge 0}\binom{l+r}{l}\binom{n}{l+r+1}.
$$
We can compute the latter sum using the identity
\begin{align*}
\sum_{r\ge 0}\binom{l+r}{l}\binom{n}{l+r+1}u^r &= \frac{1}{l!}\frac{d^l}{du^l}\left(\frac{(1+u)^n - 1}{u}\right) \\
&= (-1)^{l+1}u^{-(l+1)} + \sum_{i=0}^l(-1)^{l-i}\binom{n}{i}(1+u)^{n-i}u^{-(l-i+1)};
\end{align*}
the first equality above follows from the binomial theorem, and the second equality follows from the product rule for the derivative. Setting $u=1$ gives
$$
h_1 = l-n + (-1)^{l+1} + \sum_{i=0}^l(-1)^{l-i}\binom{n}{i}2^{n-i}.
$$
For example, when $l=0$ we obtain $\eulerian{n}{1} = h_1 = 2^n - n - 1$.
\end{example}

\section{Normalized flow}\label{sec:sperner}

\noindent In this section we prove \cref{thm:sperner_intro}, which states that $R_{n,l}$ is rank-log-concave and strongly Sperner. We prove the former in \cref{cor:whitney_numbers}, and the latter in \cref{thm:flow} using Harper's notion of a {\itshape normalized flow} \cite{harper74}.

\subsection{Background}\label{sec:background_sperner}

We provide some background on unimodal and log-concave sequences, the strongly Sperner property, and normalized flows, following \cite{stanley12,engel97,harper74}.
\begin{definition}\label{def:log_concave}
Let $s = (s_1, \dots, s_d)$ be a sequence of nonnegative real numbers. We say that $s$ is {\itshape unimodal} if for some $1 \le j \le d$, we have
$$
s_1 \le \cdots \le s_{j-1} \le s_j \ge s_{j+1} \ge \cdots \ge s_d.
$$
We say that $s$ is {\itshape log-concave} if
$$
s_{i-1}s_{i+1} \le s_i^2 \quad \text{ for } 2 \le i \le d-1.
$$

One can verify that if $s$ is a log-concave sequence of nonnegative real numbers and has no internal zeros, then $s$ is unimodal. We also observe that the entry-wise product of two log-concave sequences is log-concave.
\end{definition}

\begin{definition}\label{def:whitney}
Let $P$ be a finite graded poset of rank $d-1$, with ranks labeled from $1$ to $d$. For $1 \le r \le d$, the {\itshape $r$th Whitney number of the second kind} $W_r$ is defined to be the number of elements of $P$ of rank $r$. In terms of the flag $f$-vector, we have $W_r = \alpha_{\{r\}}$. We say that $P$ is {\itshape rank-unimodal} (respectively, {\itshape rank-log-concave}) if the sequence $(W_1, \dots, W_r)$ is unimodal (respectively, log-concave). We observe that if $P$ is rank-log-concave, then it is rank-unimodal.
\end{definition}

For example, from \cref{fig:R31} we see that for $R_{3,1}$, we have $(W_1, W_2) = (3,2)$. For general $R_{n,l}$, we can read off $W_r$ from \cref{prop:flagfvector}:
\begin{proposition}\label{cor:whitney_numbers}
The Whitney numbers of the second kind of $R_{n,l}$ (with ranks labeled from $1$ to $n-l$) are
$$
W_r = \binom{l+r-1}{l}\binom{n}{l+r} \quad \text{ for } 1 \le r \le n-l.
$$
In particular, $R_{n,l}$ is rank-log-concave.
\end{proposition}

\begin{proof}
The formula for $W_r$ follows by taking $S = \{r\}$ in \cref{prop:flagfvector}. The sequence $(W_1, \dots, W_{n-l})$ is log-concave because it is the entry-wise product of the log-concave sequences
\begin{gather*}
\bigg(\binom{l+r-1}{l}\bigg)_{r=1}^{n-l} \quad \text{ and } \quad \bigg(\binom{n}{l+r}\bigg)_{r=1}^{n-l}.\qedhere
\end{gather*}

\end{proof}

We now introduce the (strongly) Sperner property and normalized flows. Recall that an {\itshape antichain} in a poset is a subset of pairwise incomparable elements.
\begin{definition}\label{def:sperner}
Let $P$ be a finite graded poset of rank $d-1$, with ranks labeled from $1$ to $d$. Given $j \ge 1$, we say that $P$ is {\itshape $j$-Sperner} if the maximum size of a union of $j$ antichains is realized by taking the $j$ largest ranks, i.e.,
$$
|A_1 \cup \cdots \cup A_j| \le \max_{1 \le r_1 < \cdots < r_j \le d}W_{r_1} + \cdots + W_{r_j} \quad \text{ for all antichains } A_1, \dots, A_j \subseteq P.
$$
We say that $P$ is {\itshape Sperner}\footnote{The term is so named because Sperner showed that the Boolean algebra $\bool{n}$ has the Sperner property \cite{sperner28}. In fact, $\bool{n}$ is strongly Sperner (cf.\ \cite[Example 4.6.2]{engel97}).} if $P$ is $1$-Sperner, and we say that $P$ is {\itshape strongly Sperner} if $P$ is $j$-Sperner for all $j\ge 1$.
\end{definition}

\begin{definition}[{\cite{harper74}; \cite[p.\ 150]{engel97}}]\label{def:flow}
Let $P$ be a finite graded poset of rank $d-1$, with ranks labeled from $1$ to $d$. A {\itshape normalized flow} is an edge labeling $f$ (of the edges of the Hasse diagram of $P$) taking values in $\R_{\ge 0}$, such that the following conditions hold for $1 \le r \le d-1$:\footnote{The original definition also requires that the sum of $f(x\lessdot y)$ over all cover relations $x\lessdot y$ between ranks $r$ and $r+1$ equals $1$. Given an $f$ satisfying \ref{NF1} and \ref{NF2}, we can achieve this additional constraint by rescaling all such $f(x\lessdot y)$ by the same appropriate positive constant (depending on $r$).}\vspace*{2pt}
\begin{enumerate}[label={(NF\arabic*)},leftmargin=48pt,itemsep=5pt]
\item\label{NF1} $\displaystyle\sum_{y,\, x\lessdot y}f(x\lessdot y)$ is the same positive number for all $x\in P$ of rank $r$; and
\item\label{NF2} $\displaystyle\sum_{x,\, x\lessdot y}f(x\lessdot y)$ is the same positive number for all $y\in P$ of rank $r+1$.
\end{enumerate}

\end{definition}

Harper \cite[Theorem p.\ 55]{harper74} showed that if $P$ admits a normalized flow, then $P$ is strongly Sperner. In fact, it follows from work of Kleitman \cite{kleitman74} (see \cite[Theorem 4.5.1]{engel97}) that such a $P$ satisfies the stronger {\itshape LYM inequality}:
$$
\sum_{x\in A}\frac{1}{W_{\rankfun{x}}} \le 1 \quad \text{ for all antichains } A.
$$

\subsection{Construction of the normalized flow}\label{sec:flow_construction}

We define a normalized flow on $R_{n,l}$. Our definition will manifestly satisfy \ref{NF1}, and we will then check carefully that \ref{NF2} holds.
\begin{definition}\label{def:flow_construction}
Let $0 \le l < n$. We define an edge labeling $f$ on $R_{n,l}$, with label set $\mathbb{R}_{\ge 0}$, as follows. Let $x = (A_1, \dots, A_{l+1})\in R_{n,l}$, and let $a\in [n]\setminus(A_1\cup\cdots\cup A_{l+1})$. Consider all elements $y\gtrdot x$ obtained from $x$ by adding $a$ to some block; there are exactly $1$ or $2$ such $y$. There is a unique such $y$ if and only if $a < \max(A_1)$ or $a > \min(A_{l+1})$, in which case, we set
$$
f(x\lessdot y) := 1.
$$
Otherwise, we have $\max(A_i) < a < \min(A_{i+1})$ for some $1 \le i \le l$. We can add $a$ either to the $i$th block or to the $(i+1)$th block, forming, say, $y_1$ and $y_2$, respectively. We then set
$$
f(x\lessdot y_1) := \frac{|A_1\cup\cdots\cup A_i|}{|A_1\cup\cdots\cup A_{l+1}|} \quad \text{ and } \quad f(x\lessdot y_2) := \frac{|A_{i+1}\cup\cdots\cup A_{l+1}|}{|A_1\cup\cdots\cup A_{l+1}|}.
$$
Note that in either case, given $x$ and $a$, the sum of $f(x\lessdot y)$ over all $y$ obtained from $x$ by adding $a$ to some block equals $1$.
For example, see \cref{fig:R31_flow} and \cref{fig:R92_flow}.
\end{definition}

\begin{theorem}\label{thm:flow}
The edge labeling of $R_{n,l}$ in \cref{def:flow_construction} is a normalized flow. In particular, $R_{n,l}$ is strongly Sperner.
\end{theorem}

\begin{proof}
Fix $1 \le r \le n-l-1$. Let $x\in R_{n,l}$ have rank $r$, so that $x = (A_1, \dots, A_{l+1})$ with $|A_1| + \cdots + |A_{l+1}| = l+r$. Then by construction, we have
$$
\sum_{y,\, x\lessdot y}f(x\lessdot y) = n-l-r,
$$
which is positive and depends only on $r$. Therefore \ref{NF1} holds.

\begin{figure}[t]
\begin{center}
    \begin{tikzpicture}[xscale=2.4,yscale=1.2]
            \node (pm0) at (-2,0) {$(\{1\},\{2\})$};
            \node (p0m) at (0,0) {$(\{1\},\{3\})$};
            \node (0pm) at (2,0) {$(\{2\},\{3\})$};
            \node (pmm) at (-1,2) {$(\{1\},\{2,3\})$};      
            \node (ppm) at (1,2) {$(\{1,2\},\{3\})$};

            \draw (pm0.north) -- (pmm.south);
            \draw (p0m.north) -- (ppm.south);
            \draw (p0m.north) -- (pmm.south);
            \draw (0pm.north) -- (ppm.south);
            
            \node at (-1.575,1) [color=blue,fill=white] {$1$};
            \node at (-0.525,1) [color=blue,fill=white] {$\frac{1}{2}$};
            \node at (0.525,1) [color=blue,fill=white] {$\frac{1}{2}$};
            \node at (1.575,1) [color=blue,fill=white] {$1$};
    \end{tikzpicture}
\end{center}
    \caption{The normalized flow on $R_{3,1}$ defined in \cref{def:flow_construction}.}
    \label{fig:R31_flow}
\end{figure}
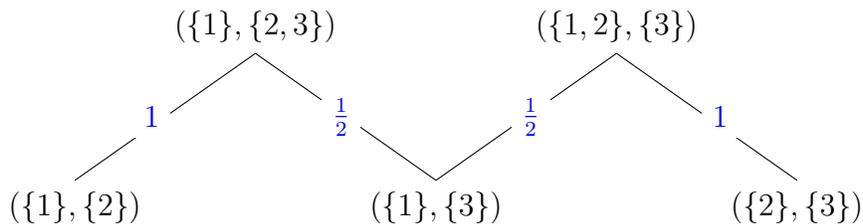
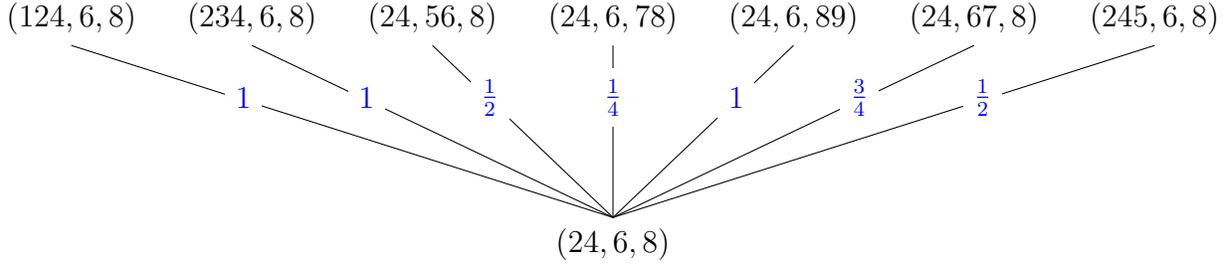
\begin{figure}[t]
\begin{center}
    \begin{tikzpicture}[xscale=1.2,yscale=1.5]
            \pgfmathsetmacro{\da}{0.35};
            
            \node (min) at (0,-2) {$(24,6,8)$};
            \node (x1) at (-6,0) {$(124,6,8)$};
            \node (x2) at (-4,0) {$(234,6,8)$};
            \node (x3) at (-2,0) {$(24,56,8)$};
            \node (x4) at (0,0) {$(24,6,78)$};
            \node (x5) at (2,0) {$(24,6,89)$};
            \node (x6) at (4,0) {$(24,67,8)$};
            \node (x7) at (6,0) {$(245,6,8)$};

            \draw (min.north) -- (x1.south);
            \draw (min.north) -- (x2.south);
            \draw (min.north) -- (x3.south);
            \draw (min.north) -- (x4.south);
            \draw (min.north) -- (x5.south);
            \draw (min.north) -- (x6.south);
            \draw (min.north) -- (x7.south);
            
            \node at (-6.3+6.3*\da,-2*\da) [color=blue,fill=white] {$1$};
            \node at (-4.2+4.2*\da,-2*\da) [color=blue,fill=white] {$1$};
            \node at (-2.1+2.1*\da,-2*\da) [color=blue,fill=white] {$\frac{1}{2}$};
            \node at (0,-2*\da) [color=blue,fill=white] {$\frac{1}{4}$};
            \node at (2.1-2.1*\da,-2*\da) [color=blue,fill=white] {$1$};
            \node at (4.2-4.2*\da,-2*\da) [color=blue,fill=white] {$\frac{3}{4}$};
            \node at (6.3-6.3*\da,-2*\da) [color=blue,fill=white] {$\frac{1}{2}$};
    \end{tikzpicture}
\end{center}
\caption{The element $(\{2,4\},\{6\},\{8\})\in R_{9,2}$, the elements covering it, and the values of the normalized flow defined in \cref{def:flow_construction}.}
    \label{fig:R92_flow}
\end{figure}

Now we prove \ref{NF2}. Let $y\in R_{n,l}$ have rank $r+1$, and write $y = (B_1, \dots, B_{l+1})$. Let $s_i := |B_i|$ for $1 \le i \le l+1$, so that $s_1 + \cdots + s_{l+1} = l+r+1$. Note that the elements $x\lessdot y$ are precisely those obtained from $y$ by selecting some block $i$ ($1 \le i \le l+1$) with $s_i \ge 2$, and removing some element $b$ of $B_i$. The value $f(x\lessdot y)$ is determined according to the following three cases:
\begin{enumerate}[label=(\roman*), leftmargin=*, itemsep=4pt]
\item if $b = \min(B_i)$, then $f(x\lessdot y) = \frac{s_i + \cdots + s_{l+1}-1}{l+r}$;
\item if $\min(B_i) < b < \max(B_i)$, then $f(x\lessdot y) = 1$; and
\item if $b = \max(B_i)$, then $f(x\lessdot y) = \frac{s_1 + \cdots + s_i-1}{l+r}$.
\end{enumerate}
The sum of the values $f(x\lessdot y)$, over all $b$ in all three cases above (with $y$ and $i$ fixed), equals
$$
\frac{s_i + \cdots + s_{l+1} - 1}{l+r} + (s_i - 2) + \frac{s_1 + \cdots + s_i - 1}{l+r} = \frac{l+r+1}{l+r}(s_i-1).
$$
Note that this formula also gives the desired sum (i.e.\ $0$) when $s_i = 1$. Therefore we obtain
$$
\sum_{x,\, x\lessdot y}f(x\lessdot y) = \sum_{i=1}^{l+1}\frac{l+r+1}{l+r}(s_i-1) = \frac{r(l+r+1)}{l+r},
$$
which is positive and depends only on $r$. This completes the proof.
\end{proof}

\section{The poset \texorpdfstring{$P_{n,l}$}{P(n,l)}}\label{sec:P}

\begin{figure}[t]
\begin{center}
    \begin{tikzpicture}[xscale=1.2]
            \node (p00) at (-3,-2) {$(+,0,0)$};
            \node (0p0) at (0,-2) {$(0,+,0)$};
            \node (00p) at (3,-2) {$(0,0,+)$};
            \node (pm0) at (-5,0) {$(+,-,0)$};
            \node (p0m) at (-3,0) {$(+,0,-)$};
            \node (pp0) at (-1,0) {$(+,+,0)$};
            \node (p0p) at (1,0) {$(+,0,+)$};
            \node (0pm) at (3,0) {$(0,+,-)$};
            \node (0pp) at (5,0) {$(0,+,+)$};
            \node (pmm) at (-3,2) {$(+,-,-)$};
            \node (ppm) at (0,2) {$(+,+,-)$};
            \node (ppp) at (3,2) {$(+,+,+)$};
     
            \draw (p00.north) -- (pm0.south);
            \draw (p00.north) -- (p0m.south);
            \draw (p00.north) -- (pp0.south);
            \draw (p00.north) -- (p0p.south);
            \draw (0p0.north) -- (pm0.south);
            \draw (0p0.north) -- (pp0.south);
            \draw (0p0.north) -- (0pm.south);
            \draw (0p0.north) -- (0pp.south);
            \draw (00p.north) -- (p0m.south);
            \draw (00p.north) -- (p0p.south);
            \draw (00p.north) -- (0pm.south);
            \draw (00p.north) -- (0pp.south);
            \draw (pm0.north) -- (pmm.south);
            \draw (p0m.north) -- (ppm.south);
            \draw (p0m.north) -- (pmm.south);
            \draw (pp0.north) -- (ppm.south);
            \draw (pp0.north) -- (ppp.south);
            \draw (p0p.north) -- (ppp.south);
            \draw (0pm.north) -- (ppm.south);
            \draw (0pp.north) -- (pmm.south);
            \draw (0pp.north) -- (ppp.south);
    \end{tikzpicture}
\end{center}
    \caption{The Hasse diagram of the poset $P_{3,1}$.}
    \label{fig:P31}
\end{figure}
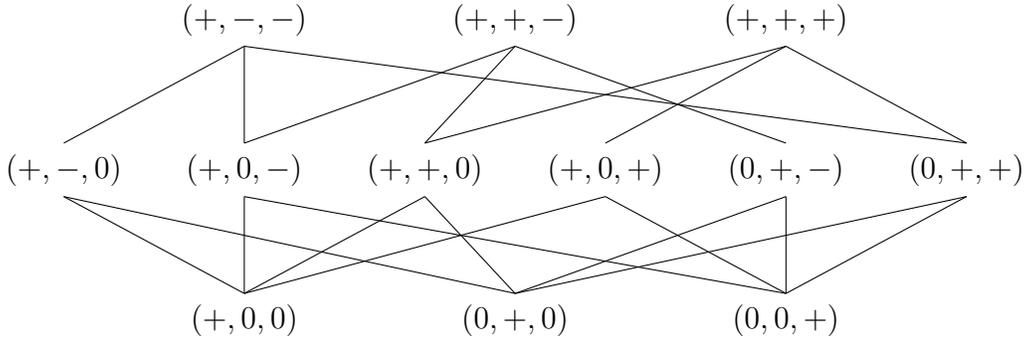
\noindent In this section we consider the poset $P_{n,l}$. Recall that $P_{n,l}$ is the poset of projective sign vectors of length $n$ with at most $l$ sign changes, under the relation \eqref{sign_vector_definition} (see \cref{fig:P31}).

It is natural to ask which properties of $R_{n,l}$ carry over to $P_{n,l}$. First we consider shellability. Since $P_{n,0} = R_{n,0}$, by e.g.\ \cref{thm:shelling_intro}, we have that $\hat{P}_{n,0}$ is EL-shellable. We can also verify directly that $\hat{P}_{2,1}$ is EL-shellable. We claim that in the remaining cases, $\hat{P}_{n,l}$ is not shellable. Indeed, if it were shellable, then the order complex of $P_{n,l}$ would be homeomorphic to a sphere or a closed ball of dimension $n-1$ \cite[Proposition 4.3]{bjorner84}. On the other hand, Machacek \cite{machacek} showed that the order complex of $P_{n,l}$ is homotopy equivalent to $\mathbb{RP}^l$, which is homeomorphic to the sphere $S^1$ when $l=1$, and is not homotopy equivalent to a sphere or a closed ball when $l \ge 2$.

We now show that $P_{n,l}$, like $R_{n,l}$, is rank-log-concave. We will use the following lemma, which appeared in talk slides of Mani \cite{mani09}. We give a proof following an argument of Semple and Welsh \cite[Example 2.2]{semple_welsh08}, who showed that a similar sequence is log-concave.
\begin{lemma}\label{lem:sum_log_concave}
Let $l\in\mathbb{N}$. Then the sequence $(s_1, s_2, \dots)$ is log-concave, where
\begin{gather*}
s_r := \sum_{i=0}^l\binom{r-1}{i} \quad \text{ for } r \ge 1.
\end{gather*}

\end{lemma}

\begin{proof}
We must show that $s_{r+1}s_{r+3} \le s_{r+2}^2$ for $r \ge 0$. Using Pascal's identity $\binom{n}{i} = \binom{n-1}{i} + \binom{n-1}{i-1}$, we get
$$
s_{r+2} = 2s_{r+1} - \binom{r}{l} \quad \text{ and } \quad s_{r+3} = 4s_{r+1} - 3\binom{r}{l} - \binom{r}{l-1}.
$$
Therefore we can rewrite the inequality $s_{r+1}s_{r+3} \le s_{r+2}^2$ as
$$
\binom{r}{l}\bigg(s_{r+1} - \binom{r}{l}\bigg) \le \binom{r}{l-1}s_{r+1}.
$$
This follows by summing the inequalities
\begin{gather*}
\binom{r}{l}\binom{r}{i-1} \le \binom{r}{l-1}\binom{r}{i} \quad \text{ for } 0 \le i \le l.\qedhere
\end{gather*}

\end{proof}

\begin{theorem}\label{thm:P_unimodal}
Let $0 \le l < n$. The Whitney numbers of the second kind of $P_{n,l}$ (with ranks labeled from $1$ to $n$) are
$$
W_r = \binom{n}{r}\sum_{i = 0}^l \binom{r-1}{i} \quad \text{ for } 1 \le r \le n.
$$
The sequence $(W_1, \dots, W_n)$ is log-concave, i.e., $P_{n,l}$ is rank-log-concave.
\end{theorem}

\begin{proof}
The set of elements of $P_{n,l}$ is the disjoint union of $R_{n,i}$ for $0 \le i \le l$, where rank $s$ of $R_{n,i}$ appears in $P_{n,l}$ in rank $s+i$. Therefore by \cref{cor:whitney_numbers}, we have
$$
W_r(P_{n,l}) = \sum_{i=0}^{\min(l,r-1)}W_{r-i}(R_{n,i}) = \binom{n}{r}\sum_{i = 0}^l \binom{r-1}{i} \quad \text{ for } 1 \le r \le n.
$$
This proves the formula for $W_r$. Now note that $(W_1, \dots, W_n)$ is the product of the two sequences
$$
\bigg(\binom{n}{r}\bigg)_{r=1}^{n} \quad \text{ and } \quad \bigg(\sum_{i=0}^l\binom{r-1}{i}\bigg)_{r=1}^{n}.
$$
We can verify that the first sequence is log-concave, and the second sequence is log-concave by \cref{lem:sum_log_concave}. Therefore $(W_1, \dots, W_n)$ is log-concave.
\end{proof}

We conjecture that $P_{n,l}$, like $R_{n,l}$, is Sperner:
\begin{conjecture}\label{conj:sperner}
For $0 \leq l < n$, the poset $P_{n,l}$ is Sperner.
\end{conjecture}
We have verified that \cref{conj:sperner} holds for all $0 \le l < n \le 8$. We also show that it holds when $l$ equals $0$, $1$, or $n-1$:

\begin{proposition}\label{prop:sperner_conj}
The posets $P_{n,0}$, $P_{n,1}$, and $P_{n,n-1}$ admit a normalized flow, and hence are strongly Sperner.
\end{proposition}

\begin{proof}
For $P_{n,0} = R_{n,0}$, this follows from \cref{thm:flow}. For $P_{n,n-1}$, the constant function $1$ is a normalized flow. This is because $P_{n,n-1}$ is {\itshape biregular}, i.e., any two elements of $P_{n,n-1}$ of the same rank have the same up-degree and the same down-degree in the Hasse diagram.

Finally, we construct a normalized flow $f$ on $P_{n,1}$, similar to the one defined on $R_{n,1}$ in \cref{def:flow_construction}. Let $x\in P_{n,1}$, and let $a\in [n]$ such that $x_a = 0$. Consider the elements covering $x$ obtained by changing entry $a$ to either $+$ or $-$; there are exactly one or two of them. If there is one such element, say $y$, we set $f(x\lessdot y) := 1$. If there are two such elements, say $y_1$ and $y_2$, we set $f(x\lessdot y_i) := \frac{1}{2}$ for $i = 1,2$. Then if $x$ has rank $r$ (with $1 \le r \le n-1$), there are exactly $n-r$ possible values of $a$, so
$$
\sum_{y,\, x\lessdot y}f(x\lessdot y) = n-r.
$$
This is positive and depends only on $r$, which proves \ref{NF1}.

Now we verify that \ref{NF2} holds. Let $1 \le r \le n-1$, and let $y\in P_{n,l}$ have rank $r+1$. Given $a\in [n]$ such that $y_a \neq 0$, let $x\lessdot y$ be obtained from $y$ by changing entry $a$ to $0$, and let $z$ be the sign vector obtained from $y$ by flipping entry $a$ (from $+$ to $-$ or vice versa). If $z$ has at most one sign change, then $f(x\lessdot y) = \frac{1}{2}$, while if $z$ has at least two sign changes, then $f(x\lessdot y) = 1$. We observe that the first case occurs for exactly $2$ values of $a$, while the second case occurs for the remaining $r-1$ values of $a$. Therefore
$$
\sum_{x,\, x\lessdot y}f(x\lessdot y) = 2(\textstyle\frac{1}{2}) + (r-1) = r,
$$
which is positive and depends only on $r$. This proves \ref{NF2}.
\end{proof}

\bibliographystyle{alpha}
\bibliography{ref}

\end{document}